\documentclass[11pt]{article}
\usepackage{booktabs}
\usepackage[sort,nocompress]{cite}
\usepackage{fullpage,multirow}
\usepackage{graphicx}
\usepackage{amsmath,amssymb,amsthm}
\usepackage{latexsym}
\usepackage{subfigure}
\usepackage{crop}
\usepackage{framed}
\usepackage{multirow}
\usepackage{bm}
\usepackage{bbm}
\usepackage{enumerate}
\usepackage{framed}
\usepackage{url}
\usepackage{array}
\usepackage{paralist}
\usepackage{diagbox}
\usepackage[colorlinks = true, pdfstartview = FitV, linkcolor = blue, citecolor = blue, urlcolor = blue]{hyperref}

\usepackage[many]{tcolorbox}
\theoremstyle{plain}
\newtheorem{assumption}{Assumption}
\newtheorem{condition}{Condition}
\newtheorem{corollary}{Corollary}

\newtheorem{theorem}{Theorem}
\tcolorboxenvironment{assumption}{sharp corners,boxrule=0.4pt,colback=white,before skip=\topsep,after skip=\topsep}
\tcolorboxenvironment{condition}{sharp corners,boxrule=0.4pt,colback=white,before skip=\topsep,after skip=\topsep}
\tcolorboxenvironment{corollary}{sharp corners,boxrule=0.4pt,colback=white,before skip=\topsep,after skip=\topsep}
\tcolorboxenvironment{lemma}{sharp corners,boxrule=0.4pt,colback=white,before skip=\topsep,after skip=\topsep}
\tcolorboxenvironment{theorem}{sharp corners,boxrule=0.4pt,colback=white,before skip=\topsep,after skip=\topsep}

\newcounter{comment}

\usepackage{xspace}

\usepackage{enumitem}
\setlist[enumerate,1]{leftmargin=*,wide=0em, label = {\bfseries \arabic*.}}
\setlist[itemize,1]{leftmargin=*,wide=0em}

\usepackage{pifont}
\renewcommand {\AA}  { {\mathbf{A}} }

\newcommand {\HH}  { {\mathbf{H}} }

\newcommand {\VV}  { {\mathbf{V}} }

\newcommand{\eye}{\mathbf{I}}

\newcommand {\bgg}  { {\bf g} }

\newcommand {\xx}  { {\bf x} }

\newcommand {\yy}  { {\bf y} }

\newcommand {\pp}  { {\bf p} }

\newcommand {\vv}  { {\bf v} }
\newcommand {\ww}  { {\bf w} }

\newcommand {\bb}  { {\bf b} }

\newcommand {\zero}  { {\bf 0} }

\renewcommand{\vec}[1]{\ensuremath{\mathbf{#1}}}
\newcommand{\grad}{\ensuremath {\vec \nabla}}

\newcommand{\defeq}{\triangleq}

\newcommand\bbR{\ensuremath{\mathbb{R}}} 

\usepackage{listings} 

\definecolor{mygreen}{rgb}{0,0.6,0}
\definecolor{mygray}{rgb}{0.5,0.5,0.5}
\definecolor{mymauve}{rgb}{0.58,0,0.82}

\usepackage{algorithm}
\usepackage{algorithmic}

\begin{document}

\title{DINO: Distributed Newton-Type Optimization Method}
\author{
    Rixon Crane\footnote{School of Mathematics and Physics, University of Queensland, Australia. Email: r.crane@uq.edu.au} 
    \qquad 
    Fred Roosta\footnote{School of Mathematics and Physics, University of Queensland, Australia, and International Computer Science Institute, Berkeley, USA. Email: fred.roosta@uq.edu.au}
}
\maketitle

\begin{abstract}
We present a novel communication-efficient Newton-type algorithm for finite-sum optimization over a distributed computing environment. Our method, named DINO, overcomes both theoretical and practical shortcomings of similar existing methods. Under minimal assumptions, we guarantee global sub-linear convergence of DINO to a first-order stationary point for general non-convex functions and arbitrary data distribution over the network. Furthermore, for functions satisfying Polyak-Lojasiewicz (PL) inequality, we show that DINO enjoys a linear convergence rate. Our proposed algorithm is practically parameter free, in that it will converge regardless of the selected hyper-parameters, which are easy to tune. Additionally, its sub-problems are simple linear least-squares, for which efficient solvers exist. Numerical simulations demonstrate the efficiency of DINO as compared with similar alternatives.
\end{abstract}

\section{Introduction}\label{introduction}

Consider the generic finite-sum optimization problem
\begin{equation}\label{eq: objective function}
    \min_{\ww \in \bbR^{d}} \bigg\{ f(\ww) \triangleq \frac{1}{m} \sum_{i = 1}^{m} f_{i}(\ww) \bigg\},
\end{equation}
in a centralized distributed computing environment comprised of one central driver machine communicating to $m$ worker machines, where each worker $i$ only has access to $f_i$.
A common application of this problem is where each worker $i$ has access to a portion of $n$ data points $\{\xx_1,\ldots,\xx_n\}$, indexed by a set $S_i\subseteq\{1,\ldots,n\}$, with
\begin{equation}\label{eq: objective function 2}
    f_i(\ww) = \frac{|S_i|}{n}\sum_{j\in S_i}\ell_j(\ww;\xx_j),
\end{equation}
where $\ell_i$ is a loss function corresponding to $\xx_i$ and parameterized by $\ww$.
For example, such settings are popular in industry in the form of federated learning when data is collected and processed locally, which increases computing resources and facilitates data privacy \cite{agarwal2018cpsgd}.
Another example is in big data regimes, where it is more practical, or even necessary, to partition and store large datasets on separate machines \cite{DiSCO}.
Distributing machine learning model parameters is also becoming a necessity as some highly successful models now contain billions of parameters, such as GPT-2 \cite{radford2019language, lee2014model}.

The need for distributed computing has motivated the development of many frameworks.
For example, the popular machine learning packages PyTorch \cite{PyTorch} and Tensorflow \cite{TensorFlow} contain comprehensive functionality for distributed training of machine learning models.
Despite many benefits to distributed computing, there are significant computational bottlenecks, such as those introduced through bandwidth and latency \cite{shamir2014distributed, li2014communication, wangni2018gradient}.
Frequent transmission of data is expensive in terms of time and physical resources \cite{DiSCO}.
For example, even transferring data on a local machine can be the main contributor to energy consumption \cite{shalf}.

With the bottleneck of communication, there has recently been significant focus on developing communication-efficient distributed optimization algorithms. 
This is particularly apparent in regards to popular first-order methods, such as stochastic gradient descent (SGD) \cite{haddadpour2019trading, ivkin2019communication, vogels2019powersgd, basu2019qsparse, teng2019leader, zheng2019communication}.
As first-order methods solely rely on gradient information, which can often be computed easily in parallel, they are usually straightforward to implement in a distributed setting.
However, their typical inherent nature of performing many computationally inexpensive iterations, which is suitable and desirable on a single machine, leads to significant data transmission costs and ineffective utilization of increased computing resources in a distributed computing environment \cite{GIANT}.

\subsection*{Related Work}
In contrast to first-order methods, second-order methods perform more computation per iteration and, as a result, often require far fewer iterations to achieve similar results. 
In distributed settings, these properties directly translate to more efficient utilization of the increased computing resources and far fewer communications over the network. 
Motivated by this potential, several distributed Newton-type algorithms have recently been developed, most notably DANE \cite{DANE}, DiSCO \cite{DiSCO}, InexactDANE and AIDE \cite{AIDE}, GIANT \cite{GIANT}, and DINGO \cite{crane2019dingo}.

While each of these second-order distributed methods have notable benefits, they all come with disadvantages that limit their applicability. 
DiSCO and GIANT are simple to implement, as they involve sub-problems in the form of linear systems. 
Whereas, the sub-problems of InexactDANE and AIDE involve non-linear optimization problems and their hyper-parameters are difficult and time consuming to tune.
DiSCO and GIANT rely on strong-convexity assumptions, and GIANT theoretically requires particular function form and data distribution over the network in \eqref{eq: objective function 2}.
In contrast, InexactDANE and AIDE are applicable to non-convex objectives.
DINGO’s motivation is to not require strong-convexity assumptions, i.e., it converges for invex problems \cite{mishra2008invexity}, and still be easy to use in practice, i.e., simple to tune hyper-parameters and linear-least squares sub-problems. 
DINGO achieves this by optimizing the norm of the gradient as a surrogate function. 
Thus, it may converge to a local maximum or saddle point in non-invex problems. 
Moreover, the theoretical analysis of DINGO is limited to exact update.

\begin{table*}[t]
    \caption{Comparison of problem class, function form and data distribution. 
    DINGO is suited to invex problems in practice, as it may converge to a local maximum or saddle point in non-invex problems \cite{crane2019dingo}.
    This is a modified table from \cite{crane2019dingo}.
    }\label{table: 1}
    \centering
    \small
    \begin{tabular}{@{} lccc @{}}
        \toprule
        & Problem Class & Function Form & Data Distribution \\
        \midrule
        \textbf{DINO} & Non-Convex & Any & Any \\
        \textbf{DINGO} \cite{crane2019dingo} & Invex & Any & Any \\
        \textbf{GIANT} \cite{GIANT} & Strongly Convex & $\ell_j(\ww; \xx_{j}) = \psi_j(\langle \ww, \xx_{j} \rangle) + \gamma \|\ww\|^{2}$ in \eqref{eq: objective function 2} & $|S_{i}| > d$ in \eqref{eq: objective function 2} \\
        \textbf{DiSCO} \cite{DiSCO} & Strongly Convex & Any & Any \\
        \textbf{InexactDANE} \cite{AIDE} & Non-Convex & Any & Any \\
        \textbf{AIDE} \cite{AIDE} & Non-Convex & Any & Any \\
        \bottomrule
    \end{tabular}
\end{table*}

\begin{table*}[t]
    \caption{
    Comparison of number of hyper-parameters (under exact update) and communication rounds per iteration (under inexact update).
    Additional hyper-parameters may be introduced under inexact update, such as in the solver used for the non-linear optimization sub-problems of InexactDANE and AIDE.
    We assume DINO, DINGO and GIANT use two communication rounds per iteration for line-search.
    This is a modified table from \cite{crane2019dingo}.
    }\label{table: 2}
    \centering
    \small
    \begin{tabular}{@{} lccc @{}}
        \toprule
        & Number of {Hyper-parameters} & Communication Rounds Per \\
        & (Under Exact Update) & Iteration (Under Inexact Update) \\
        \midrule
        \textbf{DINO}                           & $2$ & $6$     \\
        \textbf{DINGO} \cite{crane2019dingo}    & $2$ & $4$ to $8$ \\
        \textbf{GIANT} \cite{GIANT}             & $0$ & $6$     \\
        \textbf{DiSCO} \cite{DiSCO}             & $0$ & $2+2\cdot(\text{sub-problem iterations})$     \\
        \textbf{InexactDANE} \cite{AIDE}        & $2$ & $4$     \\
        \textbf{AIDE} \cite{AIDE}               & $3$ & $4\cdot(\text{inner InexactDANE iterations})$ \\
        \bottomrule
    \end{tabular}
\end{table*}

\subsection*{Contributions}
We present a novel communication-efficient distributed second-order optimization algorithm that combines many of the above-mentioned desirable properties. 
Our algorithm is named DINO, for ``\textbf{DI}stributed \textbf{N}ewton-type \textbf{O}ptimization method".
Our method is inspired by the novel approach of DINGO, which allowed it to circumvent various theoretical and practical issues of other methods.
However, unlike DINGO, we minimize \eqref{eq: objective function} directly and our analysis involves less assumptions and is under inexact update; see Tables \ref{table: 1} and \ref{table: 2} for high-level algorithm properties.

A summary of our contributions is as follows.

\begin{enumerate}
    \item The analysis of DINO is simple, intuitive, requires very minimal assumptions and can be applied to arbitrary non-convex functions.
    Namely, by requiring only Lipschitz continuous gradients $\grad f_i$, we show global sub-linear convergence for general non-convex \eqref{eq: objective function}.
    Recall that additional assumptions are typically required for the analysis of second-order methods.
    Such common assumptions include strong convexity, e.g., in \cite{roosta2019sub}, and Lipschitz continuous Hessian, e.g., in \cite{xu2017newton}, which are both required for GIANT.
    Although the theory of DINGO does not require these, it still assumes additional unconventional properties of the Hessian and, in addition, is restricted to invex problems, as a strict sub-class of general non-convex models.
    Furthermore, in our analysis, we don't assume specific function form or data distribution for applications of the form \eqref{eq: objective function 2}. For example, this is in contrast to GIANT, which is restricted to loss functions involving linear predictor models and specific data distributions.
    
    \item DINO is practically parameter free, in that it will converge regardless of the selected hyper-parameters. 
    This is in sharp contrast to many first-order methods. 
    The hyper-parameters of InexactDANE and AIDE require meticulous fine-tuning and these are sensitive to the given application.
    DINO is simple to tune and performs well across a variety of problems without modification of the hyper-parameters.
    
    \item The sub-problems of DINO are simple. 
    Like DINGO, the sub-problems of our method are simple linear least-squares problems for which efficient and robust direct and iterative solvers exists.
    In contrast, non-linear optimization sub-problems, such as those arising in InexactDANE and AIDE, can be difficult to solve and often involve additional hard to tune hyper-parameters.
\end{enumerate}

\subsection*{Notation and Definitions}

Throughout the paper, vectors and matrices are denoted by bold lower-case and bold upper-case letters, respectively, e.g., $ \vv $ and $ \VV $. 
We use regular lower-case and upper-case letters to denote scalar constants, e.g., $ d $  or $ L $. 
The common \textit{Euclidean inner product} is denoted by $\langle \xx,\yy \rangle = \xx^T\yy$ for $\xx,\yy\in\bbR^d$.
Given a vector $\vv$ and matrix $\AA$, we denote their vector $\ell_2$ norm and matrix \textit{spectral} norm as 
$\|\vv\|$ and $\|\AA\|$, respectively. 
The \textit{Moore--Penrose inverse} of $\AA$ is denoted by $\AA^\dagger$.
We let $\ww_t\in\mathbb{R}^{d}$ denote the point at iteration $t$. 
For notational convenience, we denote 
$\bgg_{t,i} \defeq \grad f_i(\ww_t)$,
$\bgg_{t} \defeq \grad f(\ww_t)$,
$\HH_{t,i} \defeq \grad^2 f_i(\ww_t)$ and 
$\HH_{t} \defeq \grad^2 f(\ww_t)$.
We also let
\begin{align}
\label{eq: H tilde and g tilde}
    \tilde{\HH}_{t,i} \defeq \begin{bmatrix} \HH_{t,i} \\ \phi\eye \end{bmatrix} \in\bbR^{2d\times d} 
    \quad\text{and}\quad
    \tilde{\bgg}_t \defeq \begin{pmatrix} \bgg_t \\ \zero \end{pmatrix} \in\bbR^{2d},
\end{align}
where $\phi>0$, $\eye$ is the identity matrix, and $\zero$ is the zero vector. 
We say that a \textit{communication round} is performed when the driver uses a \textit{broadcast} or \textit{reduce} operation to send or receive information to or from the workers in parallel, respectively.
For example, computing the gradient $\bgg_t$ requires two communication rounds, i.e., the driver broadcasts $\ww_t$ and then, using a reduce operation, receives $\bgg_{t,i}$ from all workers to then form $\bgg_t=\sum_{i=1}^{m}\bgg_{t,i}/m$.
\section{Derivation}\label{derivation}

In this section, we describe the derivation of DINO, as depicted in Algorithm~\ref{alg: Our Method}.
Each iteration $t$ involves computing an update direction $\pp_t$ and a step-size $\alpha_t$ and then forming the next iterate $\ww_{t+1}=\ww_t+\alpha_t\pp_t$.

\subsection*{Update Direction}
When forming the update direction $\pp_t$, computing $\tilde{\HH}_{t,i}^\dagger\tilde{\bgg}_t$ and $(\tilde{\HH}_{t,i}^T\tilde{\HH}_{t,i})^{-1}\bgg_t$ constitute the sub-problems of DINO, where $\tilde{\HH}_{t,i}$ and $\tilde{\bgg}_t$ are as in \eqref{eq: H tilde and g tilde}.
Despite these being the solutions of simple linear least-squares problems, it is still unreasonable to assume these will be computed exactly. In this light, we only require that the approximate solutions satisfy the following conditions.

\begin{condition}[Inexactness Condition]\label{condition: inexactness condition}
    For all iterations $t$, all worker machines $i=1,\ldots,m$ are able to compute approximations $\vv_{t,i}^{(1)}$ and $\vv_{t,i}^{(2)}$ of $\tilde{\HH}_{t,i}^\dagger\tilde{\bgg}_t$ and $(\tilde{\HH}_{t,i}^T\tilde{\HH}_{t,i})^{-1}\bgg_t$, respectively, that satisfy:
    \begin{subequations}\label{eqs: inexactness conditions}
        \begin{align}
            \|\tilde{\HH}_{t,i}^T\tilde{\HH}_{t,i}\vv_{t,i}^{(1)} - \HH_{t,i}\bgg_t\|
            &\leq \varepsilon_i^{(1)}\|\HH_{t,i}\bgg_t\|, \label{eq: inexactness condition 1} \\
            \|\tilde{\HH}_{t,i}^T\tilde{\HH}_{t,i}\vv_{t,i}^{(2)} - \bgg_t\|
            &\leq \varepsilon_i^{(2)}\|\bgg_t\|, \label{eq: inexactness condition 2} \\
            \langle \vv_{t,i}^{(2)}, \bgg_t \rangle 
            &> 0, \label{eq: inexactness condition 3}
        \end{align}
    \end{subequations}
    where $0\leq\varepsilon_i^{(1)},\varepsilon_i^{(2)}<1$ are constants.
\end{condition}

For practical implementations of DINO, as with DINGO, DiSCO and GIANT, we don't need to compute or store explicitly formed Hessian matrices, i.e., our implementations are Hessian-free. 
Namely, approximations of $\tilde{\HH}_{t,i}^\dagger\tilde{\bgg}_t$ and $(\tilde{\HH}_{t,i}^T\tilde{\HH}_{t,i})^{-1}\bgg_t$ can be efficiently computed using iterative least-squares solvers that only require access to Hessian-vector products, such as in our implementation described in Section~\ref{section: experiments}. 
These products can be computed at a similar cost to computing the gradient \cite{schraudolph2002fast}.
Hence, DINO is applicable to \eqref{eq: objective function} with a large dimension $d$.

Condition~\ref{condition: inexactness condition} has a significant practical benefit.
Namely, the criteria \eqref{eqs: inexactness conditions} are practically verifiable as they don't involve any unknowable terms.
Requiring $\varepsilon_i^{(1)},\varepsilon_i^{(2)}<1$ ensures that the approximations are simply better than the zero vector.
The condition in \eqref{eq: inexactness condition 3} is always guaranteed if one uses the conjugate gradient method (CG) \cite{nocedal2006numerical}, regardless of the number of CG iterations.

We now derive the update direction $\pp_t$ for iteration $t$.
Our approach is to construct $\pp_t$ so that it is a suitable descent direction of \eqref{eq: objective function}.
Namely, it satisfies $\langle \pp_t, \bgg_t \rangle \leq -\theta\|\bgg_t\|^2$, where $\theta$ is a selected hyper-parameter of DINO.
We begin by distributively computing the gradient, $\bgg_t$, of \eqref{eq: objective function} and then broadcast it to all workers. 
Each worker $i$ computes the vector $\vv_{t,i}^{(1)}$, as in \eqref{eq: inexactness condition 1}, lets $\pp_{t,i}=-\vv_{t,i}^{(1)}$ and checks the condition $\langle \vv_{t,i}^{(1)}, \bgg_t \rangle \geq \theta\|\bgg_t\|^2$. 

All workers $i$ in 
\begin{equation}\label{eq: set of Case 3* iteration indices}
    \mathcal{I}_{t}
    \defeq 
    \big\{ i=1,\ldots,m \mid \langle \vv_{t,i}^{(1)}, \bgg_t \rangle < \theta\|\bgg_t\|^2 \big\},
\end{equation}
has a local update direction $\pp_{t,i}$ that is not a suitable descent direction of \eqref{eq: objective function}, as $\langle \pp_{t,i}, \bgg_t \rangle > -\theta\|\bgg_t\|^2$.
We now enforce descent in their local update direction.
For this, we consider the problem
\begin{align}
    \min_{\pp_{t,i}} & \quad \|\HH_{t,i}\pp_{t,i}+\bgg_t\|^2 + \phi^{2}\|\pp_{t,i}\|^2 \label{eq:constrained_prob}\\
    \text{s.t.} & \hspace{10mm} \langle\pp_{t,i},\bgg_t\rangle\leq-\theta\|\bgg_t\|^2, \nonumber
\end{align}
where $\phi$ is a selected hyper-parameter of DINO as in \eqref{eq: H tilde and g tilde}. 
It is easy to see that, when $\langle \tilde{\HH}_{t,i}^\dagger\tilde{\bgg}_t, \bgg_t \rangle < \theta\|\bgg_t\|^2$, the problem \eqref{eq:constrained_prob} has the exact solution
\begin{align*}
    \pp_{t,i} 
    &= - \tilde{\HH}_{t,i}^\dagger\tilde{\bgg}_t - \lambda_{t,i}(\tilde{\HH}_{t,i}^T\tilde{\HH}_{t,i})^{-1}\bgg_t, \\
    \lambda_{t,i} 
    &= \frac{-\langle \tilde{\HH}_{t,i}^\dagger\tilde{\bgg}_t, \bgg_t \rangle + \theta\|\bgg_t\|^2}
    {\big\langle (\tilde{\HH}_{t,i}^T\tilde{\HH}_{t,i})^{-1}\bgg_t, \bgg_t \big\rangle} 
    > 0.
\end{align*}
Therefore, to enforce descent, each worker $i\in\mathcal{I}_t$ computes
\begin{align*}
    \pp_{t,i} 
    &= - \vv_{t,i}^{(1)} - \lambda_{t,i}\vv_{t,i}^{(2)}, \\
    \lambda_{t,i} 
    &= \frac{-\langle \vv_{t,i}^{(1)}, \bgg_t \rangle + \theta\|\bgg_t\|^2}
    {\langle \vv_{t,i}^{(2)}, \bgg_t \rangle} 
    > 0,
\end{align*}
where $\vv_{t,i}^{(2)}$ is as in \eqref{eq: inexactness condition 2} and \eqref{eq: inexactness condition 3}.
The term $\lambda_{t,i}$ is positive by the definition of $\mathcal{I}_t$ and the condition in \eqref{eq: inexactness condition 3}.
This local update direction $\pp_{t,i}$ has the property that $\langle \pp_{t,i}, \bgg_t \rangle = -\theta\|\bgg_t\|^2$.
Using a reduce operation, the driver obtains the update direction $\pp_t=\sum_{i=1}^{m}\pp_{t,i}/m$.
By construction, $\pp_t$ is now guaranteed to be a descent direction for \eqref{eq: objective function} satisfying $\langle \pp_t,\bgg_t \rangle \leq -\theta\|\bgg_t\|^2$.

\subsection*{Step-Size}

We use Armijo line search to compute a step-size $\alpha_t$.
Namely, we choose the largest $\alpha_t > 0 $ such that
\begin{equation}\label{eq: Armijo-type line seach}
    f(\ww+\alpha_t\pp_t) \leq f(\ww_t) + \alpha_t\rho\langle\pp_t,\bgg_t\rangle,
\end{equation}
with some constant $\rho\in(0,1)$. 
As $\pp_t$ is always a descent direction, we obtain a strict decrease in the function value.
This happens regardless of the selected hyper-parameters $\theta$ and $\phi$.
Line search can be conducted distributively in parallel with two communication rounds, such as in our implementation in Section~\ref{section: experiments}. 
DINO only transmits vectors of size linear in dimension $d$, i.e., $\mathcal{O}(d)$. 
This is an important property of distributed optimization methods and is consistent with DINGO, DiSCO, DANE, InexactDANE and AIDE.

\begin{algorithm}[t]
	\caption{DINO}\label{alg: Our Method}
    \begin{algorithmic}[1] 
    	\STATE{\textbf{input} initial point $\ww_0\in\mathbb{R}^d$, 
        	gradient tolerance $\delta\geq0$, maximum iterations $T$, 
        	line search parameter $\rho\in(0,1)$, parameter $\theta>0$ 
        	and regularization parameter $\phi>0$ as in \eqref{eq: H tilde and g tilde}.}
    	\FOR{$t=0,1,2,\ldots,T-1$}
            \STATE{Distributively compute the full gradient $\bgg_t$.}
            \IF{$\|\bgg_t\| \leq \delta$}
                \STATE{\textbf{return} $\ww_t$}
            \ELSE
                \STATE{The driver broadcasts $\bgg_t$ and, in parallel, each worker $i$ computes $\vv_{t,i}^{(1)}$ in \eqref{eq: inexactness condition 1}}.
                \STATE{In parallel, each worker $i$, such that $\langle \vv_{t,i}^{(1)}, \bgg_t \rangle \geq \theta\|\bgg_t\|^2$, lets $\pp_{t,i}=-\vv_{t,i}^{(1)}$.}
                \STATE{In parallel, each worker $i$, such that $\langle \vv_{t,i}^{(1)}, \bgg_t \rangle < \theta\|\bgg_t\|^2$, computes
                \begin{align*}
                    \pp_{t,i} 
                    &= - \vv_{t,i}^{(1)} - \lambda_{t,i}\vv_{t,i}^{(2)}, \\
                    \lambda_{t,i} 
                    &= \frac{-\langle \vv_{t,i}^{(1)}, \bgg_t \rangle + \theta\|\bgg_t\|^2}
                    {\langle \vv_{t,i}^{(2)}, \bgg_t \rangle} 
                    > 0,
                \end{align*}
                where $\vv_{t,i}^{(2)}$ is as in \eqref{eq: inexactness condition 2} and \eqref{eq: inexactness condition 3}.
                    }
                    \STATE{Using a reduce operation, the driver computes $\pp_t = \frac{1}{m}\sum_{i=1}^{m}\pp_{t,i}$.}
                \STATE{Choose the largest $\alpha_t > 0$ such that 
                    $$
                        f(\ww+\alpha_t\pp_t) \leq f(\ww_t) + \alpha_t\rho\langle\pp_t,\bgg_t\rangle.
                    $$
                    }
                \STATE{The driver computes $\ww_{t+1} = \ww_{t} + \alpha_t\pp_t$.}
            \ENDIF
        \ENDFOR
        \STATE{\textbf{return} $\ww_{T}$.}
    \end{algorithmic}
\end{algorithm}
\section{Analysis}\label{section: analysis}

In this section, we present convergence results for DINO. 
We assume that $f$, in \eqref{eq: objective function}, attains its minimum on some non-empty subset of $\mathbb{R}^d$ and we denote the corresponding optimal function value by $f^*$.
As was previously mentioned, we are able to show global sub-linear convergence under minimal assumptions. 
Specifically, we only assume that the local gradient $\grad f_i$, on each worker $i$, is Lipschitz continuous.

\begin{assumption}[Local Lipschitz Continuity of Gradient]\label{assumption: local Lipschitz continuity}
    The function $f_i$ in \eqref{eq: objective function} is twice differentiable for all $i=1,\ldots,m$.
    Moreover, for all $i=1,\ldots,m$, there exists constants $L_i\in(0,\infty)$ such that
    \begin{equation*}
        \big\|\grad f_i(\xx)-\grad f_i(\yy)\big\| \leq L_i\|\xx-\yy\|,
    \end{equation*}
    for all $\xx,\yy\in\mathbb{R}^d$.
\end{assumption}

As already mentioned, assuming Lipschitz continuous gradient is common place.
Recall that Assumption~\ref{assumption: local Lipschitz continuity} implies
\begin{equation}\label{eq: Lipschitz}
    \big\|\grad f(\xx)-\grad f(\yy)\big\| \leq L\|\xx-\yy\|,
\end{equation}
for all $\xx,\yy\in\mathbb{R}^d$, where $L\defeq\sum_{i=1}^{m}L_i/m$. This in turn gives for all $\alpha\geq0$
\begin{equation}\label{eq: smoothness corollary}
	f(\ww_t+\alpha\pp_t) 
	\leq 
	f(\ww_t) + \alpha\big\langle \pp_t,\bgg_t \big\rangle + \frac{\alpha^2L}{2}\|\pp_t\|^2.
\end{equation}

Recall, for a linear system $\AA\xx=\bb$, with non-singular square matrix $\AA$, the condition number of $\AA$ is $\kappa(\AA) \defeq \|\AA^{-1}\|\|\AA\|$. 
As $\tilde{\HH}_{t,i}$ has full column rank, the matrix $\tilde{\HH}_{t,i}^T\tilde{\HH}_{t,i}$ is invertible and, under Assumption~\ref{assumption: local Lipschitz continuity}, has condition number at most $(L_i^2+\phi^2)/\phi^2$.  
Therefore, under Condition~\ref{condition: inexactness condition} and Assumption~\ref{assumption: local Lipschitz continuity} we have
\begin{subequations}\label{eqs: inexactness conditions corollary}
    \begin{align}
        \|\vv_{t,i}^{(1)} - \tilde{\HH}_{t,i}^\dagger\tilde{\bgg}_t\|
        &\leq \varepsilon_i^{(1)}\bigg(\frac{L_i^2+\phi^2}{\phi^2}\bigg) 
        \|\tilde{\HH}_{t,i}^\dagger\tilde{\bgg}_t\|, \label{eq: v1 inexactness conditions corollary} \\
        \big\| \vv_{t,i}^{(2)} - (\tilde{\HH}_{t,i}^T\tilde{\HH}_{t,i})^{-1}\bgg_t \big\|
        &\leq \varepsilon_i^{(2)}\bigg(\frac{L_i^2+\phi^2}{\phi^2}\bigg)  
        \big\| (\tilde{\HH}_{t,i}^T\tilde{\HH}_{t,i})^{-1}\bgg_t \big\|, \label{eq: v2 inexactness conditions corollary}
    \end{align}
\end{subequations}
for all iterations $t$ and all workers $i=1,\ldots,m$.
We also have the upper bound
\begin{equation}\label{eq: H tilde dagger norm bound}
    \|\tilde{\HH}_{t,i}^\dagger\| \leq \frac{1}{\phi},
\end{equation}
for all iterations $t$ and all $i=1,\ldots,m$; see \cite{crane2019dingo} for a proof.

\begin{theorem}[Convergence of DINO]\label{theorem: convergence}
	Suppose Assumption~\ref{assumption: local Lipschitz continuity} holds and that we run Algorithm~\ref{alg: Our Method} with inexact update such that Condition~\ref{condition: inexactness condition} holds with $\varepsilon_i^{(2)} < \sqrt{\phi^2/(L_i^2+\phi^2)}$ for all $i=1,\ldots,m$. 
    Then for all iterations $t$ we have $f(\ww_{t+1}) \leq f(\ww_t) - \tau\rho\theta\|\bgg_t\|^{2}$ with constants
    \begin{subequations}\label{eq: case 3 inexact convergence}
        \begin{align}
            \tau &= \frac{2(1-\rho)\theta}{La^2}, \label{eq: convergence tau} \\
            a &= 
            \frac{1}{\phi} 
            \Bigg( 1 + \frac{1}{m}\sum_{i=1}^{m} \varepsilon_i^{(1)} \frac{L_i^2+\phi^2}{\phi^2} \Bigg) 
            + \frac{1}{m} \sum_{i=1}^{m} b_i, \label{eq: convergence a} \\
            b_i &= 
            \Bigg(
            \frac{1+\varepsilon_i^{(2)}(L_i^2+\phi^2)/\phi^2}{1-\varepsilon_i^{(2)}\sqrt{(L_i^2+\phi^2)/\phi^2}}
            \Bigg)
            \Bigg(
            \frac{1}{\phi} \bigg(1+\varepsilon_i^{(1)}\frac{L_i^2+\phi^2}{\phi^2}\bigg) + \theta
            \Bigg) 
            \sqrt{\frac{L_i^2+\phi^2}{\phi^2}}, \label{eq: convergence b}
        \end{align}
    \end{subequations}
    where $\rho, \theta$ and $\phi$ are as in Algorithm~\ref{alg: Our Method}, 
    $L_i$ are as in Assumption~\ref{assumption: local Lipschitz continuity}, 
    $L$ is as in \eqref{eq: Lipschitz},
    $\varepsilon_i^{(1)}$ are as in \eqref{eq: inexactness condition 1},
    and $\varepsilon_i^{(2)}$ are as in \eqref{eq: inexactness condition 2}.
\end{theorem}

\begin{proof}
    Recall that for iteration $t$, each worker $i\in\mathcal{I}_t$, as defined in \eqref{eq: set of Case 3* iteration indices}, computes
    \begin{align*}
        \pp_{t,i} 
        &= - \vv_{t,i}^{(1)} - \lambda_{t,i}\vv_{t,i}^{(2)}, \\
        \lambda_{t,i} 
        &= \frac{-\langle \vv_{t,i}^{(1)}, \bgg_t \rangle + \theta\|\bgg_t\|^2}
        {\langle \vv_{t,i}^{(2)}, \bgg_t \rangle} 
        > 0.
    \end{align*}
    The term $\lambda_{t,i}$ is both well-defined and positive by the definition of $\mathcal{I}_{t}$ and the condition in \eqref{eq: inexactness condition 3}.
    The inexactness condition in \eqref{eq: inexactness condition 2} implies
    \begin{align*}
        - \langle \vv_{t,i}^{(2)}, \bgg_t \rangle 
        + \big\langle (\tilde{\HH}_{t,i}^T\tilde{\HH}_{t,i})^{-1}\bgg_t,\bgg_t \big\rangle
        &=
        - \big\langle  
        \tilde{\HH}_{t,i}^T\tilde{\HH}_{t,i}\vv_{t,i}^{(2)} - \bgg_t,
        (\tilde{\HH}_{t,i}^T\tilde{\HH}_{t,i})^{-1}\bgg_t
        \big\rangle \\
        &\leq \varepsilon_i^{(2)} \|\bgg_t\| \big\|(\tilde{\HH}_{t,i}^T\tilde{\HH}_{t,i})^{-1}\bgg_t\big\|.
    \end{align*}
    By Assumption~\ref{assumption: local Lipschitz continuity}, we have
    \begin{equation*}
        \varepsilon_i^{(2)} \|\bgg_t\| \big\|(\tilde{\HH}_{t,i}^T\tilde{\HH}_{t,i})^{-1}\bgg_t\big\|
        \leq 
        \varepsilon_i^{(2)} \sqrt{\frac{L_i^2+\phi^2}{\phi^2}} 
        \big\langle (\tilde{\HH}_{t,i}^T\tilde{\HH}_{t,i})^{-1}\bgg_t,\bgg_t \big\rangle.
    \end{equation*}
    Therefore,
    \begin{equation*}
        \langle \vv_{t,i}^{(2)}, \bgg_t \rangle
        \geq 
        \Bigg( 1 - \varepsilon_i^{(2)} \sqrt{\frac{L_i^2+\phi^2}{\phi^2}} \Bigg)
        \big\langle (\tilde{\HH}_{t,i}^T\tilde{\HH}_{t,i})^{-1}\bgg_t,\bgg_t \big\rangle,
    \end{equation*}
    where the right-hand side is positive by the assumption $\varepsilon_i^{(2)} < \sqrt{\phi^2/(L_i^2+\phi^2)}$ and the condition $\|\bgg_t\| > \delta$ in Algorithm~\ref{alg: Our Method}.
    
    It follows from \eqref{eqs: inexactness conditions corollary} and \eqref{eq: H tilde dagger norm bound} that
    \begin{align*}
        \lambda_{t,i} \|\vv_{t,i}^{(2)}\|
        &\leq 
        \Bigg(\frac{1+\varepsilon_i^{(2)}(L_i^2+\phi^2)/\phi^2}{1-\varepsilon_i^{(2)}\sqrt{(L_i^2+\phi^2)/\phi^2}}\Bigg)
        \big({-\langle \vv_{t,i}^{(1)},\bgg_t \rangle} + \theta\|\bgg_t\|^2\big)
        \frac{\|(\tilde{\HH}_{t,i}^T\tilde{\HH}_{t,i})^{-1}\bgg_t\big\|}{\big\|(\tilde{\HH}_{t,i}^T)^\dagger\bgg_t\big\|^2} \\
        &\leq 
        \frac{1}{\phi} 
        \Bigg(\frac{1+\varepsilon_i^{(2)}(L_i^2+\phi^2)/\phi^2}{1-\varepsilon_i^{(2)}\sqrt{(L_i^2+\phi^2)/\phi^2}}\Bigg)
        \Bigg(
        \frac{\|\vv_{t,i}^{(1)}\| \|\bgg_t\| + \theta\|\bgg_t\|^2}{\big\|(\tilde{\HH}_{t,i}^T)^\dagger\bgg_t\big\|}
        \Bigg) \\
        &\leq
        \Bigg(\frac{1+\varepsilon_i^{(2)}(L_i^2+\phi^2)/\phi^2}{1-\varepsilon_i^{(2)}\sqrt{(L_i^2+\phi^2)/\phi^2}}\Bigg)
        \Bigg(
        \frac{1}{\phi} \bigg(1+\varepsilon_i^{(1)}\frac{L_i^2+\phi^2}{\phi^2}\bigg) + \theta
        \Bigg) \sqrt{\frac{L_i^2+\phi^2}{\phi^2}} \|\bgg_t\|.
    \end{align*}
    This and \eqref{eq: v1 inexactness conditions corollary}, and how $\pp_{t,i} = - \vv_{t,i}^{(1)}$ for $i\notin\mathcal{I}_t$, imply
    \begin{equation*}
        \|\pp_t\| 
        \leq \frac{1}{m}\bigg(\sum_{i\notin \mathcal{I}_t}\|\pp_{t,i}\| + \sum_{i\in \mathcal{I}_t}\|\pp_{t,i}\|\bigg)
        \leq \frac{1}{m}\bigg(\sum_{i=1}^{m} \|\vv_{t,i}^{(1)}\| 
        + \sum_{i\in \mathcal{I}_t} \lambda_{t,i} \|\vv_{t,i}^{(2)}\|\bigg)
        \leq a \|\bgg_t\|,
    \end{equation*}
    where $a$ is as in \eqref{eq: convergence a}.
    This and \eqref{eq: smoothness corollary} imply
    \begin{equation}\label{eq: alpha inequality 1}
        f(\ww_t+\alpha\pp_t)
        \leq f(\ww_t) + \alpha\langle \pp_t,\bgg_t \rangle + \frac{\alpha^2La^2}{2}\|\bgg_t\|^2,
    \end{equation}
    for all $\alpha\geq0$.
    
    For all $\alpha\in(0,\tau]$, where $\tau$ is as in \eqref{eq: convergence tau}, we have
    \begin{equation*}
        \frac{\alpha^2La^2}{2}\|\bgg_t\|^{2} 
        \leq \alpha(1-\rho)\theta\|\bgg_t\|^2,
    \end{equation*}
    and as $\langle \pp_t,\bgg_t \rangle \leq -\theta\|\bgg_t\|^2$, by construction, we obtain
    \begin{equation*}
        \frac{\alpha^2La^2}{2}\|\bgg_t\|^{2} 
        \leq \alpha(\rho-1)\langle\pp_t,\bgg_t\rangle.
    \end{equation*}
    From this and \eqref{eq: alpha inequality 1} we have
    \begin{align*}
        f(\ww_t+\alpha\pp_t)
        \leq f(\ww_t) + \alpha\rho\langle \pp_t,\bgg_t \rangle,
    \end{align*}
    for all $\alpha\in(0,\tau]$.
    Therefore, line-search \eqref{eq: Armijo-type line seach} will pass for some step-size $\alpha_t \geq \tau$.
    Moreover, $f(\ww_{t+1}) \leq f(\ww_t) - \tau\rho\theta\|\bgg_t\|^{2}$.
\end{proof}

Theorem~\ref{theorem: convergence} implies a global sub-linear convergence rate for DINO.
Namely, as $f^* \leq f(\ww_{t+1}) \leq f(\ww_t) - \tau\rho\theta\|\bgg_t\|^{2}$, we have
\begin{equation*}
    \sum_{k=1}^{t}\|\bgg_k\|^2 \leq \frac{f(\ww_0)-f^*}{\tau\rho\theta},
\end{equation*}
for all iterations $t$.
This implies $\lim_{t\rightarrow\infty}\|\bgg_t\|=0$ and
\begin{equation*}
    \min_{0\leq k\leq t} \big\{\|\bgg_k\|^2\big\} \leq \frac{f(\ww_0)-f^*}{t\tau\rho\theta},
\end{equation*}
for all iterations $t$.
Reducing the inexactness error $\varepsilon_i^{(1)}$ or $\varepsilon_i^{(2)}$ in Condition~\ref{condition: inexactness condition} lead to improved constants in the rate obtained in Theorem~\ref{theorem: convergence}.

The hyper-parameters $\theta$ and $\phi$ have intuitive effects on DINO. 
Increasing $\theta$ will also increase the chance of $\mathcal{I}_t$, in \eqref{eq: set of Case 3* iteration indices}, being large.
In fact, if $\mathcal{I}_t$ is empty for all iterations $t$, then, in Theorem~\ref{theorem: convergence}, the condition on $\varepsilon_i^{(2)}$ can be removed and the term $\tau$ can be improved to
\begin{equation*}
    \tau = \frac{2(1-\rho)\theta}{La^2}, \quad
    a = 
    \frac{1}{\phi} 
    \Bigg( 1 + \frac{1}{m}\sum_{i=1}^{m} \varepsilon_i^{(1)} \frac{L_i^2+\phi^2}{\phi^2} \Bigg).
\end{equation*}
The hyper-parameter $\phi$ controls the condition number of $\tilde{\HH}_{t,i}^T\tilde{\HH}_{t,i}$, which is at most $(L_i^2+\phi^2)/\phi^2$.
Increasing $\phi$ will decrease the condition number and make the sub-problems of DINO easier to solve, as can be seen in \eqref{eqs: inexactness conditions corollary}, while also causing a loss of curvature information in the update direction.
Also, the upper bound on $\varepsilon_i^{(2)}$ can be made arbitrarily close to $1$ by increasing $\phi$.
In practice, simply setting $\theta$ and $\phi$ to be small often gives the best performance.

Theorem~\ref{theorem: convergence} applies to arbitrary non-convex $\eqref{eq: objective function}$ satisfying the minimal Assumption~\ref{assumption: local Lipschitz continuity}.
Additional assumptions on the function class of \eqref{eq: objective function} can lead to improved convergence rates.
One such assumption is to relate the gradient $\grad f(\ww_t)$ to the distance of the current function value $f(\ww_t)$ to optimality $f^*$.
This can allow rates to be derived as iterates approach optimality.
A simple assumption of this type is the long-standing Polyak-Lojasiewicz (PL) inequality \cite{hamed2016linear}.
A function satisfies the PL inequality if there exists a constant $\mu>0$ such that
\begin{equation}\label{eq: PL inequality}
    f(\ww)-f^* \leq \frac{1}{\mu}\big\|\grad f(\ww)\big\|^2,
\end{equation}
for all $\ww\in\mathbb{R}^d$.
Under this inequality, DINO enjoys the following linear convergence rate.

\begin{corollary}[Convergence of DINO Under PL Inequality]\label{corollary: convergence PL}
	In addition to the assumptions of Theorem~\ref{theorem: convergence}, suppose that  
	the PL inequality \eqref{eq: PL inequality} holds and we run Algorithm~\ref{alg: Our Method}.
	Then for all iterations $t$ we have $f(\ww_{t+1})-f^* \leq (1-\tau\rho\mu\theta)\big(f(\ww_t)-f^*\big)$, 
	where $\rho$ and $\theta$ are as in Algorithm~\ref{alg: Our Method}, 
	$\tau$ is as in Theorem~\ref{theorem: convergence}, 
	and $\mu$ is as in \eqref{eq: PL inequality}.
	Moreover, for any choice of $\theta>0$ we have $0\leq 1-\tau\rho\mu\theta < 1$.
\end{corollary}

\begin{proof}
    As $f(\ww_{t+1}) \leq f(\ww_t) + \alpha_t\rho\langle \pp_t,\bgg_t \rangle$ and $\alpha_t\geq\tau$, the PL inequality, \eqref{eq: PL inequality}, implies
    \begin{equation*}
        f(\ww_{t+1}) - f(\ww_t)
        \leq \alpha_t\rho \langle \pp_t,\bgg_t \rangle
        \leq -\tau\rho\theta \|\bgg_t\|^2
        \leq -\tau\rho\mu\theta \big(f(\ww_t)-f^*\big),
    \end{equation*}
    which gives $f(\ww_{t+1})-f^* \leq (1-\tau\rho\mu\theta)\big(f(\ww_t)-f^*\big)$.

    From $\langle \pp_t,\bgg_t \rangle \leq -\theta\|\bgg_t\|^2$ and \eqref{eq: alpha inequality 1} we have
    \begin{equation}\label{eq: alpha inequality 2}
        f(\ww_t+\alpha\pp_t)
        \leq f(\ww_t) - \alpha\theta\|\bgg_t\|^2 + \frac{\alpha^2La^2}{2}\|\bgg_t\|^2,
    \end{equation}
    for all $\alpha\geq0$.
    The right-hand side of \eqref{eq: alpha inequality 2} is minimized when $\alpha=\theta/(La^2)$. 
    It has a minimum value of $f(\ww_t) - \big(\theta^2/(2La^2)\big)\|\bgg_t\|^2$, which, by \eqref{eq: alpha inequality 2}, must be at least $f^*$.
    This and \eqref{eq: PL inequality} imply
    \begin{equation*}
        \frac{\theta^2}{2La^2}\|\bgg_t\|^2 \leq f(\ww_t) - f^* \leq \frac{1}{\mu}\|\bgg_t\|^2,
    \end{equation*}
    which gives $\theta \leq \sqrt{2La^2/\mu}$.
    Therefore,
    \begin{equation*}
        \tau\rho\mu\theta
        =  \frac{2\rho\mu(1-\rho)\theta^2}{La^2}
        \leq 4\rho(1-\rho)
        \leq 1,
    \end{equation*}
    which implies $0\leq 1-\tau\rho\mu\theta < 1$.
\end{proof}
\begin{table*}[t]
    \caption{Number of iterations completed in one hour.
    In one column, the driver and all five worker machines are running on one node.
    In another, they are running on their own instances in the distributed computing environment over AWS.
    The performance improvement, going from one node to AWS, is also presented.
    We use the code from \cite{crane2019dingo} to replicate their AWS results.
    }\label{table: iterations}
    \vspace{10pt}
    \centering
    \small
    \begin{tabular}{@{} lccc @{}}
        \toprule
        & Number of Iterations & Number of Iterations & Change When \\
        & (Running on one Node) & (Running over AWS) & Going to AWS \\
        \midrule
        \textbf{DINO}                       & $ 3    $ & $ 12   $ & $ +300\% $ \\
        \textbf{DINGO} \cite{crane2019dingo}& $ 3    $ & $ 12   $ & $ +300\% $ \\
        \textbf{GIANT} \cite{GIANT}         & $ 4    $ & $ 18   $ & $ +350\% $ \\
        \textbf{DiSCO} \cite{DiSCO}         & $ 7    $ & $ 19   $ & $ +171\% $ \\
        \textbf{InexactDANE} \cite{AIDE}    & $ 213  $ & $ 486  $ & $ +128\% $ \\
        \textbf{AIDE} \cite{AIDE}           & $ 214  $ & $ 486  $ & $ +127\% $ \\
        \textbf{SGD} \cite{SGD}             & $ 1743 $ & $ 1187 $ & $  -32\% $ \\
        \bottomrule
    \end{tabular}
\end{table*}
\begin{figure*}[t]%
    \centering
    \subfigure{%
    \includegraphics[height=20pt]{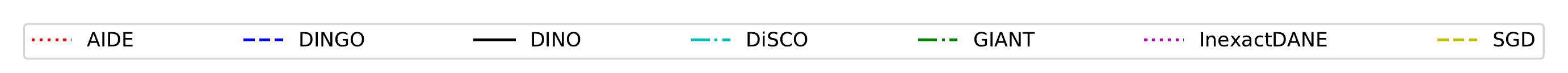}}%
    \\
    \addtocounter{subfigure}{-1}
    \subfigure{%
    \includegraphics[height=4.8cm]{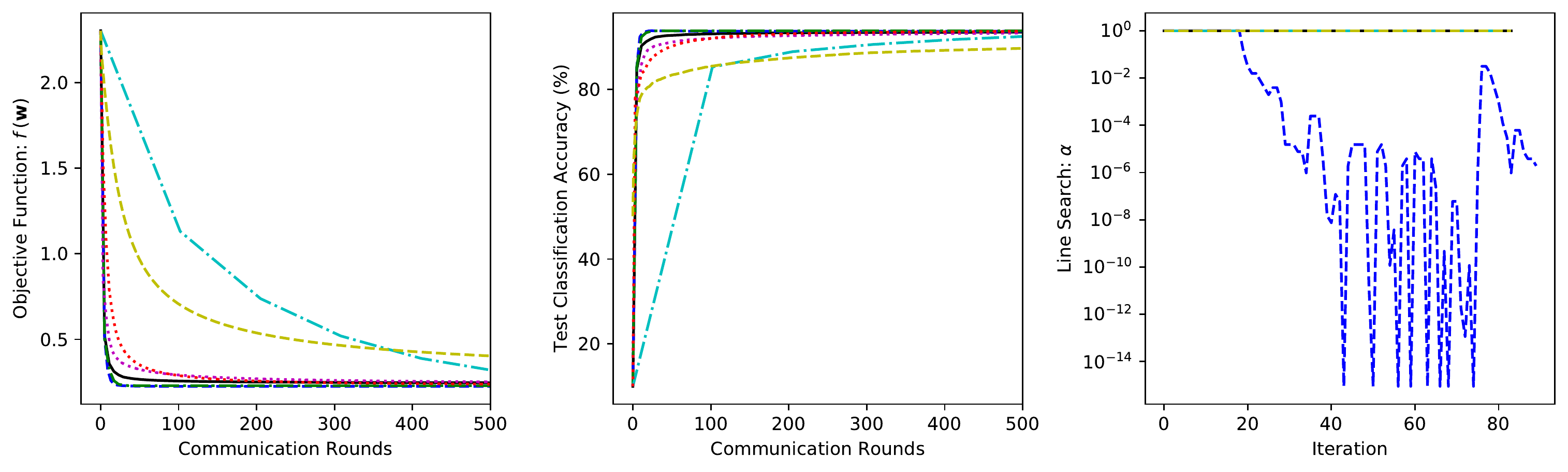}}%
    \caption{Softmax regression problem on the EMNIST Digits dataset.
    SVRG, in InexactDANE, and SGD both have a learning rate of $10^{-1}$ and AIDE has $\tau=1$.
    Here, we have five worker nodes, i.e., $m=5$ in \eqref{eq: objective function}.
    }
    \label{fig: softmax EMNIST}
\end{figure*}

The PL inequality has become widely recognized in both optimization and machine learning literature \cite{LeiYunwen2019SGDf}. 
The class of functions satisfying the condition contains strongly-convex functions as a sub-class and contains functions that are non-convex \cite{hamed2016linear}. 
This inequality has shown significant potential in the analysis of over-parameterized problems and is closely related to the property of interpolation \cite{bassily2018exponential,vaswani2019painless}.
Functions satisfying \eqref{eq: PL inequality} are a subclass of invex functions \cite{Newton-MR}.
Invexity is a generalization of convexity and was considered by DINGO.
Linear MLP and some linear ResNet are known to satisfy the PL inequality \cite{furusho2019skipping}.
\section{Experiments}\label{section: experiments}

\begin{figure*}[t]%
    \centering
    \subfigure{%
    \includegraphics[height=20pt]{figures/legend.pdf}}%
    \\
    \addtocounter{subfigure}{-1}
    \subfigure{%
    \includegraphics[height=4.8cm]{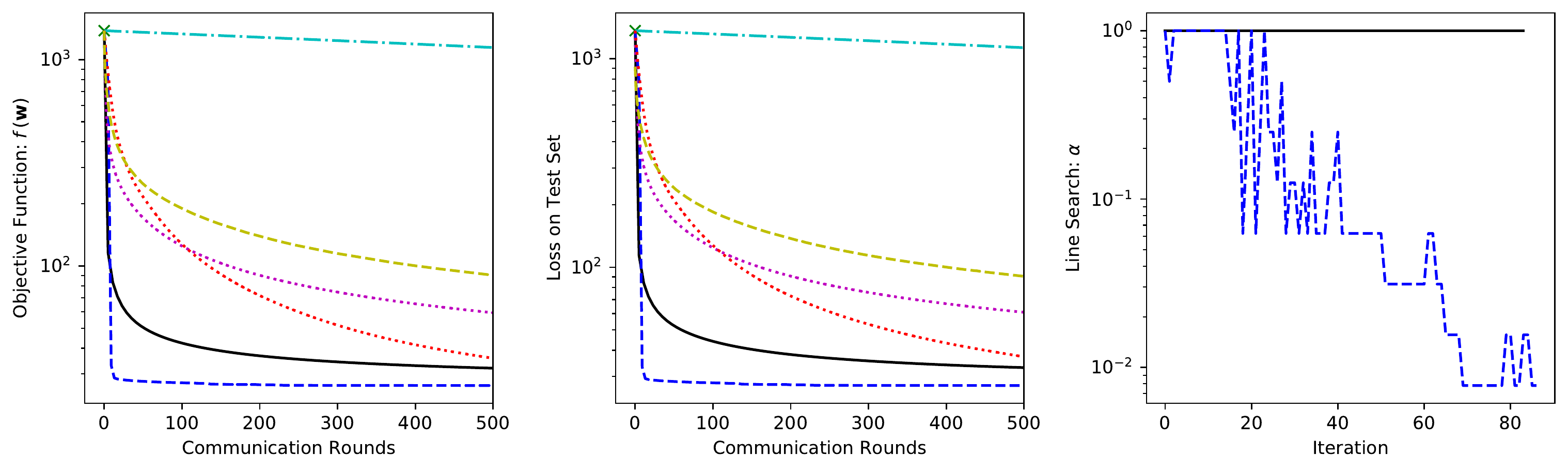}}%
    \caption{Non-linear least-squares problem on the CIFAR10 dataset.
    SVRG, in InexactDANE, and SGD have a learning rate of $10^{-4}$ and $10^{-3}$, respectively, and AIDE has $\tau=100$.
    GIANT failed immediately.
    Here, we have 50 worker nodes, i.e., $m=50$ in \eqref{eq: objective function}.
    }
    \label{fig: NLLS CIFAR10 50 workers}
\end{figure*}

In this section, we examine the empirical performance of DINO in comparison to the, previously discussed, distributed second-order methods DINGO, DiSCO, GIANT, InexactDANE and AIDE. 
We also compare these to synchronous SGD \cite{SGD}.
In all experiments, we consider \eqref{eq: objective function} with \eqref{eq: objective function 2}, where $S_1,\ldots,S_m$ partition $\{1,\ldots,n\}$ with each having equal size $n/m$.
In Table~\ref{table: iterations} and Figure~\ref{fig: softmax EMNIST}, we compare performance on the strongly convex problem of softmax cross-entropy minimization with regularization on the EMNIST Digits dataset.
In Figure~\ref{fig: NLLS CIFAR10 50 workers}, we consider the non-convex problem of non-linear least-squares without regularization on the CIFAR10 dataset with $\ell_j(\ww;\xx_j) = \big(y_j-\log(1+\exp{\langle \ww,\xx_j \rangle})\big)^2$ in \eqref{eq: objective function 2}, where $y_j$ is the label of $\xx_j$.
Although GIANT and DiSCO require strong convexity, we run them on this problem and indicate, with an ``$\times$" on the plot, if they fail.

We first describe some of the implementation details. 
The sub-problems of DINO, DINGO, DiSCO, GIANT and InexactDANE are limited to 50 iterations, without preconditioning. 
For DINO, we use the well known iterative least squares solvers LSMR \cite{LSMR} and CG to approximate $\vv_{t,i}^{(1)}$ and $\vv_{t,i}^{(2)}$ in Algorithm~\ref{alg: Our Method}, respectively.
For DINO, and DINGO as in \cite{crane2019dingo}, we use the hyper-parameters $\theta=10^{-4}$ and $\phi=10^{-6}$.
For DINO, DINGO and GIANT we use distributed backtracking line-search to select the largest step-size in $\{1,2^{-1},\ldots,2^{-50}\}$ that passes, with an Armijo line-search parameter of $10^{-4}$.
For InexactDANE, we set the hyper-parameters $\eta=1$ and $\mu=0$, as in \cite{AIDE}, which gave high performance in \cite{DANE}.
We also use the sub-problem solver SVRG \cite{SVRG} and report the best learning rate from $\{10^{-5},\ldots,10^5\}$.
We let AIDE call only one iteration of InexactDANE, which has the same parameters as the stand-alone InexactDANE algorithm.
We also report the best acceleration parameter, $\tau$ in \cite{AIDE}, from $\{10^{-5},\ldots,10^5\}$.
For SGD, we report the best learning rate from $\{10^{-5},\ldots,10^5\}$ and at each iteration all workers compute their gradient on a mini-batch of $n/(5m)$ data points.

The run-time is highly dependent on the distributed computing environment, which is evident in Table~\ref{table: iterations}.
Here, we run the methods on a single node on our local compute cluster.
We also run them over a distributed environment comprised of six Amazon Elastic Compute Cloud instances via Amazon Web Services (AWS).
These instances are located in Ireland, Ohio, Oregon, Singapore, Sydney and Tokyo.
This setup is to highlight the effect of communication costs on run-time.
As can be seen in Table~\ref{table: iterations}, the second-order methods experience a notable speedup when going to the more powerful AWS setup, whereas SGD experiences a slowdown.
DINO, DINGO and GIANT performed the most local computation in our experiments and they also had the largest increase in iterations.
Similar behaviour to that in Table~\ref{table: iterations} can also be observed for the non-linear least-squares problem.

In Figures~\ref{fig: softmax EMNIST} and \ref{fig: NLLS CIFAR10 50 workers}, we compare the number of communication rounds required to achieve descent. 
We choose communication rounds as the metric, as time is highly dependent on the network.
DINO is competitive with the other second-order methods, which all outperform SGD.
Recall that InexactDANE and AIDE are difficult to tune.
Between Figures~\ref{fig: softmax EMNIST} and \ref{fig: NLLS CIFAR10 50 workers}, notice the significant difference in the selected learning rate for SVRG of InexactDANE and the acceleration parameter of AIDE.
Meanwhile, DINO and DINGO have consistent performance, despite not changing hyper-parameters.

In the non-convex problem in Figure~\ref{fig: NLLS CIFAR10 50 workers}, GIANT fails immediately as CG fails on all 50 worker nodes. 
DiSCO does not fail and has poor performance. 
This suggests that locally around the initial point $\ww_0$, the full function $f$ is exhibiting convexity, while the local functions $f_i$ are not.
Moreover, in Figure~\ref{fig: NLLS CIFAR10 50 workers} the dimension $d$, which is $3072$, is larger than the number of training samples, $1000$, on each worker node.

\section{Conclusion and Future Work}\label{section: future}
In the context of centralized distributed computing environment, we present a novel distributed Newton-type method, named DINO, which enjoys several advantageous properties. DINO is guaranteed to converge under minimal assumption, its analysis is simple and intuitive, it is practically parameter free, and it can be applied to arbitrary non-convex functions and data distributions. Numerical simulations highlight some of these properties.
 
The following is left for future work. 
First, characterizing the relationship between the hyper-parameters $\theta$ and $\phi$ of DINO.
As was discussed, they have intuitive effects on the algorithm and they are easy to tune.
However, there is a non-trivial trade off between them that will be explored in future work.
Second, analysing the connection between DINO and over-parameterized problems.
Finally, extending the theory of DINO to alternative forms of line search, such as having each worker perform local line search and then aggregating this information in a way that preserves particular properties.
\bibliography{bibliography}

\bibliographystyle{unsrt}

\end{document}